\documentclass[10pt,a4paper,reqno]{amsart}

\usepackage{amsmath}
\usepackage{mathptmx}
\usepackage{amssymb}
\usepackage{amsthm}
\usepackage{txfonts}
\usepackage{microtype}
\usepackage[usenames,dvipsnames]{xcolor}
\usepackage{enumitem}
\usepackage[colorlinks=true,linkcolor=Blue,citecolor=Blue,urlcolor=Blue]{hyperref}
\usepackage[capitalise]{cleveref}
\usepackage{mathrsfs}
\usepackage{calc}  
\usepackage[foot]{amsaddr}
\usepackage{tikz}
\usetikzlibrary{decorations.pathreplacing}

\theoremstyle{plain}
\newtheorem{theorem}{Theorem}[section]
\newtheorem{lemma}[theorem]{Lemma}
\newtheorem{proposition}[theorem]{Proposition}
\newtheorem{corollary}[theorem]{Corollary}
\newtheorem{example}[theorem]{Example}

\theoremstyle{definition}
\newtheorem{definition}[theorem]{Definition}
\newtheorem{remark}[theorem]{Remark}

\begin{document}

\allowdisplaybreaks

\title{Regularity of aperiodic minimal subshifts}

\author[F.\ Dreher]{F.\ Dreher}
\address[F.\ Dreher]{Independent Scholar}
%
\author[M.\ Kesseb\"ohmer]{M.\ Kesseb\"ohmer}
\address[M.\ Kesseb\"ohmer]{FB 3 -- Mathematik, Universit\"at Bremen, Bibliothekstr. 1, 28359 Bremen, Germany}
%
\author[A.\ Mosbach]{A.\ Mosbach}
\address[A.\ Mosbach]{FB 3 -- Mathematik, Universit\"at Bremen, Bibliothekstr. 1, 28359 Bremen, Germany}
%
\author[T.\ Samuel]{T.\ Samuel}
\address[T.\ Samuel]{Mathematics Department, California Polytechnic State University, San Luis Obispo, CA, USA}
\email[Corresponding author]{ajsamuel@calpoly.edu}
\author[M.\ Steffens]{M.\ Steffens}
\address[M.\ Steffens]{FB 3 -- Mathematik, Universit\"at Bremen, Bibliothekstr. 1, 28359 Bremen, Germany}


\maketitle

\begin{abstract}
At the turn of this century Durand, and Lagarias and Pleasants established that key features of minimal subshifts (and their higher-dimensional analogues) to be studied are linearly repetitive, repulsive and power free.  Since then, generalisations and extensions of these features, namely \mbox{$\alpha$-repetitive}, \mbox{$\alpha$-repulsive} and \mbox{$\alpha$-finite} ($\alpha \geq 1$), have been introduced and studied.  We establish the equivalence of \mbox{$\alpha$-repulsive} and \mbox{$\alpha$-finite} for general subshifts over finite alphabets.  Further, we studied a family of aperiodic minimal subshifts stemming from Grigorchuk's infinite $2$-group $G$.  In particular, we show that these subshifts provide examples that demonstrate $\alpha$-repulsive (and hence $\alpha$-finite) is not equivalent to \mbox{$\alpha$-repetitive}, for $\alpha > 1$.  We also give necessary and sufficient conditions for these subshifts to be \mbox{$\alpha$-repetitive}, and \mbox{$\alpha$-repulsive} (and hence \mbox{$\alpha$-finite}).  Moreover, we obtain an explicit formula for their complexity functions from which we deduce that they are uniquely ergodic. 
\end{abstract}

\section{Introduction}

Aperiodic subshifts over finite alphabets play a vital role in various branches of mathematics, physics, and computer science.  The theory of aperiodic order is a relatively young field of mathematics, which has attracted considerable attention in recent years, see for instance \cite{AR:1991,BG:2013,BaMo:2000,BFMS:2002,FGJ:2015,KLS:2011,KS:2012,Mo:1997,MH:1940,P:1998,S:2015}.  It has grown rapidly over the past three decades; on the one hand, due to the experimental discovery of physical solid substances, called quasicrystals, exhibiting such features \cite{INF:1985,SBGC:1984}; and on the other hand, due to intrinsic mathematical interest in describing the very border between crystallinity and aperiodicity.  While there is no axiomatic framework for aperiodic order, various types of order conditions have been studied, see \cite{AR:1991,BG:2013,D:2000,BFMS:2002,FGJ:2015,HKW:2015,KS:2012,Lag:1999,Lag:2002,Lag:2003,MH:1940} and references therein.  In particular, through the work of Durand \cite{D:2000}, and Lagarias and Pleasants \cite{Lag:2003} it has become apparent that key features of aperiodic minimal subshifts (and their higher-dimensional analogues) to be studied are linearly repetitive, repulsive and power free.  Generalisations and extensions of these characteristics, namely $\alpha$-repetitive, $\alpha$-repulsive and $\alpha$-finite ($\alpha \geq 1$), were recently introduced in \cite{GKMSS16}.  Indeed, we have that $1$-repetitive is equialent to aperiodic and linearly repetitive, that $1$-repulsive implies repulsive, and that $1$-finite is equivalent to power free.  

For $\alpha \geq 1$, a subshift $Y$ which is $\alpha$-repetitive roughly means that the maximum return time (with respect to the left-shift map) of an infinite word in $Y$ to a cylinder set $U \subset Y$ generated by a finite word $u$ is of the order $\lvert u \rvert^\alpha$; $\alpha$-repulsive loosely means that if $W$ is a factor of an infinite word in $Y$ and if $w \neq W$ is a prefix and a suffix of $W$, then the overlap of these two appearances of $w$ in $W$ is at most of the order $\lvert w \rvert - \lvert w \rvert^{1/\alpha}$; and $\alpha$-finite roughly means that if $n$ is the largest natural number such that the $n$-fold concatenation of a finite word $u$ is a factor of an infinite word in $Y$, then $n$ is at most of the order $\lvert u \rvert^{\alpha -1}$.

In \cite{GKMSS16,KS:2012}, for Sturmian subshifts with slope $\theta$ and for $\alpha \geq 1$, it was shown that the characteristics $\alpha$-repetitive, $\alpha$-repulsive and $\alpha$-finite are equivalent. Here, links between regularity of spectral metrics built from noncommutative representations (spectral triples), aperiodic behaviour of the subshift and the Diophanitine properties of $\theta$ were obtained.

Here, we address the following question.  For an arbitrary subshift and for $\alpha \geq 1$, which of the order conditions $\alpha$-repetitive, $\alpha$-repulsive and $\alpha$-finite are equivalent?

We prove that, for $\alpha \geq 1$, a subshift is $\alpha$-repulsive if and only if it is $\alpha$-finite (\Cref{thm:equivalence}).  However, for $\alpha > 1$, we establish that $\alpha$-repetitive is not necessarily equivalent to $\alpha$-repulsive, and hence, nor $\alpha$-finite (\Cref{thm:G-alpha-free,thm:G-alpha-repetative}).  This latter result is provided by a class of subshifts  stemming  from  Grigorchuk's infinite $2$-group \mbox{$G$ -- the} first known group of intermediate growth introduced by Grigorchuk \cite{GRIGO:1984,G:84b}   (see also \cite{G:84a}, where a general class of groups, denoted by $G_{\omega}$, of intermediate growth is introduced). They have been studied, for instance, by Bon \cite{Bon201567}, Grigorchuk, Lenz and Nagnibeda \cite{GLN:16,GLN:16b}, and Lenz and Sell \cite{LenzSell:2016}.  These subshifts are determined by an infinite sequence $l = (l_{i})_{i \in \mathbb{N}}$ of natural numbers  and we  refer to them as $l$-Grigorchuk subshifts.

We show that $l$-Grigorchuk subshifts are aperiodic and minimal (\Cref{prop:minimality,cor:Gri-aperiodic}). Additionally, we establish necessary and sufficient conditions for these new subshifts to be $\alpha$-repetitive and $\alpha$-repulsive, and hence, $\alpha$-finite (\Cref{thm:G-alpha-free,thm:G-alpha-repetative}).  More precisely, we prove that an $l$-Grigorchuk subshift is $\alpha$-repulsive  (and hence $\alpha$-finite) if and only if
\begin{align*}
\limsup_{n \to \infty} \left\lvert l_{n+1} + (1 - \alpha) \sum_{i = 1}^{n} l_{i} \right\rvert <\infty,
\end{align*}
and that an $l$-Grigorchuk subshift is $\alpha$-repetitive if and only if 
\begin{align*}
\limsup_{n \to \infty} \left\lvert l_{n+2} + l_{n+1} + (1 - \alpha) \sum_{i = 1}^{n} l_{i} \right\rvert < \infty.
\end{align*}
We also obtain an explicit formula in terms of the sequence $l$ for the complexity function of an \mbox{$l$-Grigorchuk} subshift (\Cref{thm:complexity}), from which we are able to deduce that an \mbox{$l$-Grigorchuk} subshift is uniquely ergodic (\Cref{Cor:unique_ergodic}). Indeed, we show that there exist at most two and at least one right special word per length.  We would like to emphasis that, independently, Lenz and Sell \cite{LenzSell:2016} have obtained an explicit formula for the repetitive and complexity functions of an $l$-Grigorchuk subshift.  Moreover, they have also computed an explicit formula for the palindromic complexity function.  Further, in the case that $l$ is the constant one sequence, results concerning the complexity function have been obtained in \cite{GLN:16,GLN:16b}.

When $l$ is the constant one sequence, the resulting $l$-Grigorchuk subshift is intimately related to Lysenok group presentation of Grigorchuk's infinite $2$-group $G$.  By studying this subshift, very recently \cite{GLN:16,GLN:16b} the spectral type of the Laplacian on the Schreier graphs describing the action of Grigorchuk's infinite $2$-group $G$ on the boundary of the infinite binary rooted tree were determined and it has been shown that it is different in the isotropic and anisotropic cases.  In fact, the spectrum is shown to be a Cantor set of Lebesgue measure zero in the anisotropic case, whereas it consists of one or two intervals in the isotropic case.  Here (\Cref{sec:sec4}), we implicitly associate to a given $l$-Grigorchuk subshift a group, investigating properties of such groups and if the results of \cite{GLN:16,GLN:16b} can be extended to encompass our setting, we believe, would be a worthwhile and fruitful venture.

\subsection*{Outline} In the next section, we present key definitions and results concerning subshifts and define $\alpha$-repetitive, $\alpha$-repulsive and $\alpha$-finite.  In \Cref{sec:GR} we state and prove the equivalence of $\alpha$-repulsive and $\alpha$-finite for arbitrary subshifts over a finite alphabet.  We conclude with \Cref{sec:4}, which is divided into five parts.  The first part (\Cref{sec:sec4}) is concerned with introducing and defining $l$-Grigorchuk subshifts as well as stating some of their basic properties.  In \Cref{sec:section4.2,sec:section4.3} we provide necessary and sufficient conditions on a sequence $l$ which ensures that the associated $l$-Grigorchuk subshift is $\alpha$-repulsive (and hence $\alpha$-finite), and $\alpha$-repetitive respectively; after which, in \Cref{sec:examples}, we present several examples of sequences $l = (l_{n})_{n \in \mathbb{N}}$ for which the associated $l$-Grigorchuk subshift is \mbox{$\alpha$-repetitive}, and $\alpha$-repulsive (and hence $\alpha$-finite) for specific values of $\alpha$.  Here, we also show that if an $l$-Grigorchuk subshift is \mbox{$\alpha$-repulsive} and hence $\alpha$-finite, then it is $\alpha^{2}$-repetitive.  In our concluding part, \Cref{sec:section4.4}, we obtain an explicit formula for the complexity function (in terms of the sequence $l$) of an $l$-Grigorchuk subshift from which we deduce that any $l$-Grigorchuk subshift is aperiodic and uniquely ergodic.

\subsection*{Acknowledgements}

The authors would like to thank Daniel Lenz and Daniel Sell for bringing the problem to their attention.  The fourth author would like to thank  AG Dynamical Systems and Geometry at Universit\"at Bremen, Fakult\"at f\"ur Mathematik und Informatik at Friedrich-Schiller-Universit\"at Jena and Institut f\"ur Mathematik at Universit\"at zu L\"ubeck for hosting him and providing a stimulating research environment while working on this project.  The last author is grateful to the Institut f\"ur Mathematik at Universit\"at zu L\"ubeck for providing a stimulating working environment during the writing of this article.

\section{Preliminary Definitions}

Here, we review the key definitions of subshifts and define three notions of aperiodic order ($\alpha$-repetitive, $\alpha$-repulsive and $\alpha$-finite, for a given $\alpha \geq 1$) first introduced for Sturmian subshifts in \cite{GKMSS16}, and which generalise and extend the order conditions often referred to as linearly repetitive, repulsive and power free.

\subsection{Subshifts}\label{sec:subshiftt_intro}

Let $\mathscr{A}$ denote a set of $m \in \mathbb{N}$ symbols called the \textit{alphabet}.  For $n \in \mathbb{N}$ we define $\mathscr{A}^{n}$ to be the set of all finite words in the alphabet $\mathscr{A}$ of length $n$, and set
\begin{align*}
\mathscr{A}^{*} \coloneqq \bigcup_{n \in \mathbb{N}_{0}} \mathscr{A}^{n},
\end{align*}
where by convention $\mathscr{A}^{0}$ is the set containing only the \textit{empty word} $\varepsilon$. We denote by $\mathscr{A}^{\mathbb{N}}$ the set of all infinite words over the alphabet $\mathscr{A}$ and equip it with the discrete product topology.  The continuous map $\sigma \colon \mathscr{A}^{\mathbb{N}} \to \mathscr{A}^{\mathbb{N}}$ defined by $\sigma( x_{1}, x_{2}, \dots ) \coloneqq ( x_{2}, x_{3}, \dots )$ is called the \textit{left-shift}.  A closed set $Y \subseteq \mathscr{A}^{\mathbb{N}}$ which is left-shift invariant (that is $\sigma(Y) = Y$) is referred to as a \emph{subshift} and the tuple $(Y, \sigma)$ forms a dynamical system.  For an infinite word $x = (x_{n})_{n \in \mathbb{N}}$ over a finite alphabet $\mathscr{A}$, we set 
\begin{align*}
\Omega(x) \coloneqq \overline{\{ \sigma^{k}(x) \colon k \in \mathbb{N}_{0} \}},
\end{align*}
where the closure is taken with respect to the discrete product topology.  We call $\Omega(x)$ the \textit{subshift generated by} $x$.  For a subshift $Y$, the dynamical system $(Y, \sigma)$ is called \textit{minimal} if for all $y \in Y$ the set $\Omega(y)$ is dense in $Y$.  If $Y$ does not contain a periodic element (that is, an element $y$, such that there exists $k \in \mathbb{N}$ with $\sigma^{k}(y) = y$), then we call $Y$ \textit{aperiodic}.

For $w = (w_{1}, \dots, w_{k})$ and $v = (v_{1}, \dots, v_{n}) \in \mathscr{A}^{*}$, we set $w v \coloneqq (w_{1}, \dots, w_{k}, v_{1}, \dots, v_{n})$, that is the \textit{concatenation} of $w$ and $v$.  For $m \in \mathbb{N}$, we denote by $v^{m}$ the $m$-fold concatenation of $v$ with itself, namely
\begin{align*}
v^{m} \coloneqq \underbrace{v v \dots v}_{m - \text{times}}.
\end{align*}
Note that  $\mathscr{A}^{*}$ together with the operation of concatenation defines a monoid with identity element $\varepsilon$. The \textit{length} of $v$ is denoted by $\lvert v \rvert$ with $\lvert \varepsilon \rvert=0$ and, for $k \leq n$ a natural number, we set $v\lvert_{k} \coloneqq (v_{1}, v_{2}, \dots, v_{k})$.  We say that a word $u \in \mathscr{A}^{*}$ is a \textit{factor} of $v$ if there exists an integer $j$ with $u = \sigma^{j-1}(v)\lvert_{\lvert u \rvert}$.  We use the same notations when $v$ is an infinite word.  The integer $j$ is refereed to as an \textit{occurrence} of $u$ in $v$.

An infinite word $x$ over a finite alphabet $\mathscr{A}$ is called \textit{recurrent} if every factor has infinitely many different occurrences in $x$.  A \textit{gap} of a factor $u$ of $x$ is an integer $k$ which is a difference between two successive occurrences of $u$ in $x$.  We say that $x$ is \textit{uniformly recurrent} if $x$ is recurrent and for each factor $u$ of $x$ there exists an upper bound for the corresponding gaps. This is equivalent to the minimality of the corresponding subshift generated by $x$, see for instance \cite{Combinatorics:2010}.

The \textit{language $\mathcal{L}(Y)$ of a subshift $Y$} is the set of all factors of the elements of $Y$.  Similarly, we define the \textit{language $\mathcal{L}(x)$ of an infinite word $x$} to be the set of all factors of $x$.  Notice, the language of $\Omega(x)$ of an infinite word $x$ is equal to the language of $x$, namely $\mathcal{L}(\Omega(\nu)) = \mathcal{L}(\nu)$.  Following convention, the empty word $\varepsilon$ is assumed to be contained in every language.  For $s \geq 2$, we call $w = (w_{1}, \dots, w_{k}) \in \mathcal{L}(Y)$ \textit{$s$-right special} if the cardinality of the set \mbox{$\{ a \in \mathscr{A} \colon (w_{1}, \dots, w_{k}, a) \in \mathcal{L}(Y) \}$} is equal to $s$.  A word is called \textit{right special} if it is $s$-right special for some $s \geq 2$.

\subsection{Notions of aperiodic order}\label{sec:aperiodic_order_intro}

We begin by stating the definition of $\alpha$-repetitive, first defined in \cite{GKMSS16} for Sturmian subshifts, which generalises the concept of linearly repetitive.

\begin{definition}\label{defn:rep_fun}
The \textit{repetitive function} $R \colon \mathbb{N} \to \mathbb{N}$ of a subshift $Y$ assigns to $r$ the smallest $r'$ such that any element of $\mathcal{L}(Y)$ with length $r'$ contains (as factors) all elements of $\mathcal{L}(Y)$ with length $r$.
\end{definition}

\begin{definition}\label{defn:repetitive}
Let $\alpha \geq 1$ be given and set
\begin{align*}
R_{\alpha} \coloneqq \limsup_{n \to \infty} \frac{R(n)}{n^{\alpha}}.
\end{align*}
A subshift $Y$ is called \textit{$\alpha$-repetitive} if $R_{\alpha}$ is finite and non-zero.
\end{definition}

\begin{remark}\label{rmk:1=linear}
If $1 \leq \alpha < \beta$ and $0 < R_{\beta} < \infty$, then $R_{\alpha} = \infty$.  Similarly, if $0 < R_{\alpha} < \infty$, then $R_{\beta} = 0$.  Also, recall that a subshift $Y$ is said to be \text{linearly repetitive}, if and only if, there exists a positive constant $K$, such that $R(n) \leq K n$, for all $n \in \mathbb{N}$.  Since aperiodicity of a subshift guarantees that the number of words of length $n$ is strictly greater than $n$, for all $n\in\mathbb{N}$, see for instance \cite{F:2002}, this yields that linearly repetitive and $1$-repetitive are equivalent for aperiodic subshifts.
\end{remark}

Next, for $\alpha \geq 1$, we state the definition of $\alpha$-repulsive, which generalises the notion of repulsive.  We recall that a subshift $Y$ is called \emph{repulsive} if the value
\begin{align*}
\ell \coloneqq \inf \left\{ \frac{\lvert W \rvert - \lvert w \rvert}{\lvert w \rvert} \colon w, W \in \mathcal{L}(Y), w \; \text{is a prefix and suffix of} \; W, \; \text{and} \; W \neq w \neq \varepsilon \right\}
\end{align*}
is non-zero.

\begin{definition}\label{defn:repulsive}
Let $\alpha \geq 1$ be given.  For a subshift $Y$ set 
\begin{align*}
\ell_{\alpha} \coloneqq \liminf_{n \to \infty} A_{\alpha, n},
\end{align*}
where for a given natural number $n \geq 2$
\begin{align*}
A_{\alpha, n} \coloneqq \inf \left\{ \frac{\lvert W \rvert - \lvert w \rvert}{\lvert w \rvert^{1/\alpha}} \colon w, W \in \mathcal{L}(Y), w \; \text{is a prefix and suffix of} \; W, \; \lvert W \rvert = n \; \text{and} \; W \neq w \neq \varepsilon \right\},
\end{align*}
and if $\ell_{\alpha}$ is finite and non-zero, then we say that $Y$ is \textit{$\alpha$-repulsive}.
\end{definition}

\begin{remark}\label{rmk:repulsive_unique}
Notice that, if $1 \leq \alpha < \beta$ and $0 < \ell_{\beta} < \infty$, then $\ell_{\alpha} = 0$.  Similarly, if $0 < \ell_{\alpha} < \infty$, then $\ell_{\beta} = \infty$.
\end{remark}

The next definition is a generalisation of the notion of a subshift being power free.  If $\alpha = 1$, then $1$-finite is equivalent to the property of being power free.

\begin{definition}\label{defn:free}
For a subshift $Y$ and for $n \in \mathbb{N}$ set
\begin{align*}
Q(n) \coloneqq \sup \{ p \in \mathbb{N} \colon \text{there exists} \; W \in \mathcal{L}(Y) \; \text{with} \; \lvert W \rvert = n \; \text{and} \; W^{p} \in \mathcal{L}(Y) \}.
\end{align*}
Let $\alpha \geq 1$ be given.  We say that the subshift $Y$ is \textit{$\alpha$-finite} if the value
\begin{align*}
Q_{\alpha} \coloneqq \limsup_{n \to \infty} \frac{Q(n)}{n^{\alpha - 1}}
\end{align*}
is non-zero and finite.  Also, for ease of notation, for a given word $v \in \mathcal{L}(Y)$, we let $Q(v)$ denote the largest integer $p$ such that $v^{p} \in \mathcal{L}(Y)$, in the case that no such $p$ exists, we set $Q(v) = \infty$.
\end{definition}

\begin{remark}\label{rmk:1=powerfree}
If $1 \leq \alpha < \beta$ and $0 < Q_{\beta} < \infty$, then $Q_{\alpha} = \infty$.  Similarly, if $0 < Q_{\alpha} < \infty$, then $Q_{\beta} = 0$.
\end{remark}

To conclude this section, we state the definition of the complexity function.

\begin{definition}
For a subshift $Y$, we define the \textit{complexity function} $p \colon \mathbb{N} \to \mathbb{N}$ of $Y$ by
\begin{align*}
p(n) \coloneqq   \operatorname{card}\{ w \in \mathcal{L}(Y) \colon \lvert w \rvert = n \} .
\end{align*}
\end{definition}

\section{General Results}\label{sec:GR}

\begin{theorem}\label{thm:equivalence}
For $\alpha \geq 1$ and $x$ an infinite word over a finite alphabet, we have that $\Omega(x)$ is $\alpha$-repulsive if and only if it is $\alpha$-finite.
\end{theorem}

\begin{proof}
Let $\alpha \geq 1$ be fixed and let $\Omega(x)$ be $\alpha$-repulsive.  Suppose that $Q_{\alpha} = \infty$.  In this case there exist sequences of natural numbers $(n_{k})_{k \in \mathbb{N}}$ and $(p_{k})_{k \in \mathbb{N}}$ satisfying 
\begin{enumerate}[itemsep=0.1em,topsep=-0.25em,label=(\roman*)]
\item $(n_{k})_{k \in \mathbb{N}}$ is increasing with $p_{k}n_{k}^{1 - \alpha} > k$, and 
\item there exists $W_{(k)} \in \mathcal{L}(x)$ with $\lvert W_{(k)} \rvert = n_{k}$ and $W_{(k)}^{p_{k}} \in \mathcal{L}(x)$.
\end{enumerate}
Thus, we have that $p_{k} > 1$, for all $k$ sufficiently large.  Since $W_{(k)}^{p_{k}-1}$ is a prefix and a suffix of $W_{(k)}^{p_{k}}$ we have that
\begin{align*}
 \frac{\lvert W_{(k)}^{p_{k}} \rvert - \lvert W_{(k)}^{p_{k}-1} \rvert}{{\lvert {W_{(k)}^{p_{k}-1}} \rvert}^{1/\alpha}}
 =\frac{\lvert W_{(k)} \rvert}{{\lvert W_{(k)} \rvert}^{1/\alpha}{(p_{(k)}-1)}^{1/\alpha}}
 =\frac{n_k}{{n_k}^{1/\alpha}(p_{k}-1)^{1/\alpha}}
 \leq\frac{2^{1/\alpha} {n_k}^{(\alpha-1)/\alpha}}{{p_{k}}^{1/\alpha}}
 <\frac{2^{1/\alpha}}{k^{1/\alpha}},
\end{align*}
for all $k$ sufficiently large.  Therefore, we have that $\ell_\alpha=0$.

Suppose that $Q_{\alpha} = 0$. For $n \in \mathbb{N}$ let $V_{(n)}, v_{(n)} \in \mathcal{L}(x)$ be such that $\lvert V_{(n)} \rvert = n$, $v_{(n)} \neq V_{(n)}$ is a prefix and suffix of $V_{(n)}$ and
\begin{align*}
\frac{\lvert V_{(n)} \rvert - \lvert v_{(n)} \rvert}{{ \lvert v_{(n)} \rvert^{\frac{1}{\alpha}}}} = A_{\alpha , n}.
\end{align*}
Since $0 < \ell_\alpha < \infty$, this means that there exists a sequence $(n_{k})_{k \in \mathbb{N}}$ of natural numbers such that $2 \lvert v_{(n_{k})} \rvert > \lvert V_{(n_k)} \rvert$, for all $k \in \mathbb{N}$.  Thus, for each $k \in \mathbb{N}$, there exists a $q_{k} \geq 2$ such that
\begin{align*}
v_{(n_{k})} = \underbrace{u_{(k)} u_{(k)} \cdots u_{(k)}}_{q_{k}-1} z_{(k)} \quad \text{and} \quad V_{(n_{k})} = \underbrace{u_{(k)} u_{(k)} \cdots u_{(k)}}_{q_{k}} z_{(k)},
\end{align*}
where $u_{({k})}, z_{({k})} \in \mathcal{L}(x)$ with $0 < \lvert z_{(k)} \rvert < \lvert u_{(k)} \rvert$.  Hence, it follows that
\begin{align}\label{eq:MT100816}
\left( \frac{\lvert V_{(n_{k})} \rvert - \lvert v_{(n_{k})} \rvert}{{\lvert v_{(n_{k})} \rvert^{\frac{1}{\alpha}}}} \right)^\alpha
 =  \frac{(\lvert V_{(n_{k})} \rvert - \lvert v_{(n_{k})} \rvert)^\alpha}{\lvert v_{(n_{k})} \rvert}
 \geq \frac{\lvert u_{(k)} \rvert^\alpha}{q_{k} \lvert u_{(k)} \rvert}
 =\frac{\lvert u_{(k)} \rvert^{\alpha -1 }}{q_{k}}
 \geq \frac{\lvert u_{(k)} \rvert^{\alpha -1}}{Q( u_{(k)} )}
 \geq \frac{\lvert u_{(k)} \rvert^{\alpha -1}}{Q( \lvert u_{(k)} \rvert )},
\end{align}
where the lengths of the $u_{(k)}$ are unbounded, as otherwise $\limsup_{k\to\infty} Q(u_{(k)}) = \infty$. 
However, since by assumption $Q_{\alpha} =0$, we have
\begin{align*}
\liminf_{n \to \infty} {\frac{n^{\alpha - 1}}{Q(n)}}=\infty.
\end{align*}
This together with \eqref{eq:MT100816} yields that $\ell_{\alpha} = \infty$.

The reverse direction follows from the proof of (3) $\Rightarrow$ (2) in \cite[Theorem 3.4]{GKMSS16}.  We note that the statement of \cite[Theorem 3.4]{GKMSS16} is in terms of Sturmian subshifts and it is assumed that $\alpha > 1$, however, the proof of (3) $\Rightarrow$ (2) holds for arbitary subshifts and for $\alpha =1$.
\end{proof} 

\begin{proposition}\label{lem:Lemma2}
Let $\alpha \geq 1$ be given and let $x$ denote an infinite word over a finite alphabet.  If $\Omega(x)$ is $\alpha$-repulsive, or equivalently $\alpha$-finite, then it is aperiodic.
\end{proposition}

\begin{proof}
We show the contrapositive. Suppose that there exists a $y \in \Omega(x)$ such that $\sigma^{k}(y) = y$, for some $k \in \mathbb{N}$.  This implies that $Q( n k ) = \infty$, for all $n \in \mathbb{N}$, and so, for all $\alpha \geq 1$ we have that $Q_\alpha = \infty$.  Therefore, the subshift $\Omega(x)$ is not $\alpha$-finite for any \mbox{$\alpha \geq 1$}.
\end{proof}

\begin{proposition}\label{Prop:lowerbound}
For an aperiodic subshift $Y$ we have that $R(n) > n Q(n)$, for all $n\in\mathbb{N}$.
\end{proposition}

\begin{proof}
Recall that aperiodicity of a subshift guarantees that the number of words of length $n$ is strictly greater than $n$, for all $n\in\mathbb{N}$, see for instance \cite{F:2002}.

Let $n \in \mathbb{N}$ be fixed.  Let $w \in \mathcal{L}(Y)$ be such that $\lvert w \rvert = n$ and $w^{Q(n)} \in \mathcal{L}(Y)$.  The word $w^{Q(n)}$ has at most $n$ different factors of length $n$.  Thus, since $\lvert w^{Q(n)} \rvert = n Q(n)$ and since $\mathcal{L}(Y)$ is aperiodic, we have that $R(n) > n Q(n)$.
\end{proof}

\begin{corollary}\label{cor:R<Q}
For an aperiodic subshift $Y$ and for $\alpha \geq 1$, we have that $R_{\alpha} \geq Q_{\alpha}$.  In particular, $R_{\alpha} = 0$ implies $Q_{\alpha} = 0$ and $Q_{\alpha} = \infty$ implies $R_{\alpha} = \infty$.
\end{corollary}

\begin{remark}\label{rmk:wquivalence}
In general it is not true that if $Q_{\alpha} = 0$, then $R_{\alpha} = 0$ and if $R_{\alpha} = \infty$, then $Q_{\alpha} = \infty$.  An infinite word $x$ in which one of the letters only occurs exactly once gives rise to a subshift $\Omega(x)$ where this occurs.  However, this subshift is not minimal.   The $l$-Grigorchuk subshifts (which we will shortly introduce in the next section) provide examples of uniquely ergodic and minimal subshifts which are \mbox{$\alpha$-finite} (or equivalently \mbox{$\alpha$-repulsive}), but not \mbox{$\alpha$-repetitive}, see \Cref{ex:example}.
\end{remark}

\section{$l$-Grigorchuk subshifts}\label{sec:4}

\subsection{$\boldsymbol{l}$-Grigorchuk subshifts}\label{sec:sec4}
The Grigorchuk subshift is a subshift associated to Grigorchuk's infinite $2$-group $G$.  The group $G$ was originally introduced in \cite{GRIGO:1984,G:84b} and is an infinite finitely generated torsion group and so belongs to the class of Burnside groups, see also \cite{G:84a}.  It has growth between polynomial and exponential, hence is amenable but not elementary amenable, see \cite{G:84a}.  This group therefore provided simultaneous answers to the question of Milnor \cite{M:1968} on existence of groups of intermediate growth, and to the question of Day \cite{Day:1957} on existence of amenable but not elementary amenable groups.  Lysenok \cite{L:85}, gave a recursive presentation of $G$ by generators and relations using a homomorphism $\kappa$, which we will shortly define, see \eqref{eq:kappa} and \eqref{eq:G_Presentation}.  It is remarkable that the homomorphism $\kappa$ serves not only to define $G$ algebraically, but also, as is shown in \cite{GLN:16}, to describe spectral properties of $G$ and to determine $G$ in terms of topological dynamics as a subgroup of the topological full group of a minimal Cantor system.
 
Following convention we consider the alphabet $\{ a, x, y, z \}$.  We define the semi-group homomorphism $\kappa \colon \{ a, x, y, z \}^{*} \to \{ a, x, y, z \}^{*}$ by
\begin{align}\label{eq:kappa}
\kappa(a) \coloneqq (a,x, a), \quad \kappa(x) \coloneqq y, \quad \kappa(y) \coloneqq z, \quad \kappa(z) \coloneqq x,
\end{align}
and for a finite word $w = (w_{1}, \dots, w_{n})$ we set $\kappa(w) \coloneqq \kappa(w_{1}) \dots, \kappa(w_{n})$.  The homomorphism $\kappa$ is defined to act on infinite words analogously.  It is known that there exists a unique infinite word $\eta \in \{ a, x, y, z \}^{\mathbb{N}}$ such that $\kappa(\eta) = \eta$, see for instance \cite{GLN:16}.  We call the subshift $\Omega(\eta)$ the \textit{Grigorchuk subshift}.  Alternatively, this subshift can be generated by the three semi-group homomorphisms $\tau_{x}$, $\tau_{y}$ and $\tau_{z}$ defined by
\begin{align*}
\tau_{\beta}(a) \coloneqq (a, \beta, a), \quad \tau_{\beta}(x) \coloneqq x, \quad \tau_{\beta}(y) \coloneqq y, \quad \tau_{\beta}z \coloneqq z,
\end{align*}
where $\beta \in \{ x, y, z \}$, and for $w = (w_{1}, \dots, w_{n})$ we set $\tau_{\beta}(w) \coloneqq {\tau}_{\beta}(w_{1}), \dots, {\tau}_{\beta}(w_{n})$.  Indeed, the word $\eta$ is the unique word with the prefix 
\begin{align*}
(\tau_{x} \circ \tau_{y} \circ \tau_{z})^{n}(a) = \underbrace{\underbrace{\tau_{x} \circ \tau_{y} \circ \tau_{z}} \circ \underbrace{\tau_{x} \circ \tau_{y} \circ \tau_{z}} \circ \dots \circ\underbrace{\tau_{x} \circ \tau_{y} \circ \tau_{z}}}_{n - \text{times}} (a),
\end{align*}
for all $n \in \mathbb{N}$.  We now introduce a more general class of subshifts based on this latter construction, which we call $l$-Grigorchuk subshifts, where each $l = (l_{k})_{k \in \mathbb{N}}$ is a sequence of natural numbers.

Let $l = (l_{k})_{k \in \mathbb{N}}$ denote a fixed sequence of natural numbers.  For $j \in \mathbb{N}$, we denote by $N(j)$ and $q(j)$ the unique integers such that
\begin{align*}
j=q(j)+ \sum_{i=1}^{N(j)-1} l_i \quad \text{with } 0 \leq q(j) < l_{N(j)}.
\end{align*} 
We define $\tau^{(j)}$ by
\begin{align}\label{eq:tauj}
\tau^{(j)}\coloneqq
\begin{cases}
\tau^{l_{1}}_{x} \circ \tau^{l_{2}}_{y} \circ \tau^{l_{3}}_{z} \circ \dots \circ  \tau^{l_{N(j)}}_{z} \circ \tau_{x}^{q(j)} & \mbox{if} \; N(j) \equiv 0 \pmod{3},\\
\tau^{l_{1}}_{x} \circ \tau^{l_{2}}_{y} \circ \tau^{l_{3}}_{z} \circ \dots \circ  \tau^{l_{N(j)}}_{x} \circ \tau_{y}^{q(j)} & \mbox{if} \; N(j) \equiv 1 \pmod{3},\\
\tau^{l_{1}}_{x} \circ \tau^{l_{2}}_{y} \circ \tau^{l_{3}}_{z} \circ \dots \circ  \tau^{l_{N(j)}}_{y} \circ \tau_{z}^{q(j)} & \mbox{if} \; N(j) \equiv 2 \pmod{3},
\end{cases}  
\end{align}
and let $\tau^{(0)}$ be the identity.  Additionally, we set
\begin{align}\label{eq:betaj}
\beta^{(j)} \coloneqq 
\begin{cases}
x & \mbox{if} \; N(j) \equiv 0 \pmod{3},\\
y & \mbox{if} \; N(j) \equiv 1 \pmod{3},\\
z & \mbox{if} \; N(j) \equiv 2 \pmod{3}.
\end{cases}
\end{align}

\begin{proposition}\label{Prop:Prop1}
For $l=(l_k)_{k \in \mathbb{N}}$, there exists a unique infinite word $\eta_{l}$ with prefix $\tau^{(j)}(a)$, for all $j \in \mathbb{N}_{0}$.
\end{proposition}

\begin{proof}
This is a consequence of the fact that, $\tau^{(j)}(a)$ is a prefix of $\tau^{(j+1)}(a)$, for all $j \in \mathbb{N}_{0}$, and, as we will see in \Cref{prop:length_of_tau}, $\lim_{j \to \infty} \lvert \tau^{(j)}(a) \rvert = \infty$.
\end{proof}

For a given sequence of natural numbers $l = (l_{k})_{k \in \mathbb{N}}$, we refer to the subshift $\Omega(\eta_{l})$ as the \textit{$l$-Grigorchuk subshift}, where $\eta_{l}$ is the unique word given in \Cref{Prop:Prop1}.  When it is clear from the context, we will write $\eta$ instead of $\eta_{l}$.  Note that the Grigorchuk subshift is an $l$-Grigorchuk subshift with $l$ equal to the constant one sequence, namely $l = (1, 1, 1, \dots )$.  By construction, for all $j \in \mathbb{N}$, we observe that $\eta$ has the form
\begin{align}\label{eq:tau_structure_eta}
\raisebox{-1em}{
\begin{tikzpicture}
\draw(0,0)--(6.75,0);
\draw [dotted] (6.75,0)--(8,0);
\foreach \x in {0,1.75,2.25,4,4.5,6.25,6.75}
\draw(\x,0.1)--(\x,-0.1);
\draw[decorate, decoration={brace}, yshift=1.5ex]  (0,0) -- node[above=0.4ex] {$\tau^{(j)}(a)$}  (1.75,0);
\draw[decorate, decoration={brace}, yshift=1.5ex]  (2.25,0) -- node[above=0.4ex] {$\tau^{(j)}(a)$}  (4,0);
\draw[decorate, decoration={brace}, yshift=1.5ex]  (4.5,0) -- node[above=0.4ex] {$\tau^{(j)}(a)$}  (6.25,0);
\draw[decorate, decoration={brace}, yshift=1.5ex]  (1.75,0) -- node[above=0.4ex] {$?$}  (2.25,0);
\draw[decorate, decoration={brace}, yshift=1.5ex]  (4,0) -- node[above=0.4ex] {$?$}  (4.5,0);
\draw[decorate, decoration={brace}, yshift=1.5ex]  (6.25,0) -- node[above=0.4ex] {$?$}  (6.75,0);
\node at (-0.4,-0.05) {$\eta = $};
\node at (8.05,-0.015) {,};
\end{tikzpicture}}
\end{align}
where the letters $x$, $y$ and $z$ occur infinitely often, in a prescribed order determined by the sequence $l$, in place of the question marks.  One can also define an $l$-Grigorchuk subshift where elements of $l$ are allowed to take the value zero, see \Cref{rmk:last_remark}.

\begin{proposition}\label{prop:length_of_tau}
For $j \in \mathbb{N}_{0}$ we have that $\displaystyle \lvert \tau^{j}(a) \vert = 2^{j + 1} - 1$.
\end{proposition}

\begin{proof}
We have that $\lvert \tau^{(0)}(a) \rvert = \lvert a \rvert = 1$.  Suppose the result holds true for some $j \in \mathbb{N}_{0}$, then
\begin{align*}
\lvert \tau^{(j+1)}(a) \rvert = \lvert \tau^{(j)}(\tau_{\beta^{(j)}}(a)) \rvert = \lvert \tau^{(j)}(a) \beta^{(j)} \tau^{(j)}(a) \rvert = 2 \lvert \tau^{(j)}(a) \rvert + 1 = 2^{(j + 1) + 1} - 1.
\end{align*}
This completes the proof.
\end{proof}

\begin{corollary}\label{cor:rpulsive_iff}
An $l$-Grigorchuk subshift is repulsive if and only if it is $1$-repulsive.
\end{corollary}

\begin{proof}
For an $l$-Grigorchuk subshift, we observe that since $\tau^{(j)}(a)$ is a prefix and suffix of $\tau^{(j+1)}(a)$ and since $\tau^{(j)}(a) \in \mathcal{L}(\eta)$, for $j \in \mathbb{N}$, by \Cref{prop:length_of_tau} we have $Q_{1} \leq 1$.  Therefore, an $l$-Grigorchuk subshift is repulsive if and only if it is $1$-repulsive.
\end{proof}

\begin{proposition}\label{prop:minimality}
An $l$-Grigorchuk subshift is minimal.
\end{proposition}

\begin{proof}
For every word $w$ in the language of $\eta$ there exists a $j \in\mathbb{N}$ such that $w$ is a factor of $\tau^{(j)}(a)$.  The structure of $\eta$, namely that given in \eqref{eq:tau_structure_eta}, yields that the gap between two successive occurrences of $w$'s is bounded, and so, $\eta$ is uniformly recurrent.  As uniformly recurrence is equivalent to minimality, see for instance \cite{Combinatorics:2010}, this completes the proof.
\end{proof}

While we do not use it in the sequel we would like to highlight the role $\kappa$ and $\tau^{(j)}$, and hence $\tau_{x}$, $\tau_{y}$ and $\tau_{z}$, play in Grigorchuk's infinite $2$-group $G$. Indeed, $\kappa$ is (a version of) the substitution used by Lysenok \cite{L:85} to obtain a presentation of $G$. More specifically, \cite{L:85} shows that
\begin{align}\label{eq:G_Presentation}
G = \langle a, x, y, z \mid 1 = a^{2} = x^{2} = y^{2} = z^{2} = \kappa^{k}((a,z)^{4}) = \kappa^{k}((a,z,a,x,a,x)^{4}), k \in \mathbb{N}_{0}  \rangle.
\end{align}
This presentation can be written using $\tau^{(j)}$, and hence $\tau_{x}$, $\tau_{y}$ and $\tau_{z}$, by using the fact that
\begin{align*}
\kappa^{j}(a,z) = \tau^{(j)}(a, \beta^{(j-1)}),
\end{align*}
and that, for all $j \in \mathbb{N}$,
\begin{align*}
\kappa^{j}(a,z,a,x,a,x) =
\tau^{(1)}(a,x) 
\tau^{(2)} (a,y) 
\tau^{(3)} (a,z) 
\tau^{(4)} (a,x) 
\dots 
\tau^{(j+1)}(a \beta^{(j)}).
\end{align*}
Here $\tau^{(j)}$ and $\beta^{(j)}$ are as defined in \eqref{eq:tauj} and \eqref{eq:betaj} with $l$ equal to the constant $1$ sequence, that is $l = (l_{i})$ with $l_{i} = 1$.

\subsection{$\boldsymbol{\alpha}$-Finite and $\boldsymbol{\alpha}$-Repulsive}\label{sec:section4.2}

\Cref{thm:G-alpha-free} below gives a necessary and sufficient condition on a given sequence of natural numbers $l$ to guarantee that the associated $l$-Grigorchuk subshift is $\alpha$-finite, which by \Cref{thm:equivalence} is equivalent to the subshift being $\alpha$-repulsive.  In particular, we obtain that an $l$-Grigorchuk subshift is $1$-finite (and hence $1$-repulsive) if and only if $l$ is a bounded sequence. Thus, as $1$-repulsive implies repulsive, if $l$ is a bounded sequence, then the associated $l$-Grigorchuk subshift is repulsive.

\begin{theorem}\label{thm:G-alpha-free}
For $\alpha \geq 1$ the following three statements are equivalent.
\begin{enumerate}[itemsep=0.1em,topsep=-0.25em,label=(\roman*)]
\item An $l$-Grigorchuk subshift is $\alpha$-repulsive.
\item An $l$-Grigorchuk subshift is $\alpha$-finite.
\item $\limsup_{n \to \infty} \lvert l_{n+1} + (1 - \alpha) \sum_{i = 1}^{n} l_{i} \rvert <\infty$.
\end{enumerate}
\end{theorem}

\begin{proof}
The result follows from \Cref{thm:equivalence} and \Cref{thm:Q-Bounds} given below.
\end{proof}

\begin{theorem}\label{thm:Q-Bounds}
For $\alpha > 1$, an $l$-Grigorchuk subshift fulfils the following equality.
\begin{align*}
Q_\alpha = \limsup_{m\to\infty} \frac{2^{l_{m+1} + 1}}{ 2^{(\alpha-1)\sum_{i=1}^{m} l_{i}}  }
\end{align*}
Moreover, we have that
\begin{align*}
\limsup_{m\to\infty} 2^{l_{m+1}+1} -1 \leq Q_{1} \leq \limsup_{m\to\infty} 2^{l_{m+1}+1}.
\end{align*}
\end{theorem}

For the proof of this result we will require the following definition and remark.

\begin{definition}
Fix a sequence $l = (l_{i})_{i \in \mathbb{N}}$ and let $\eta$ denote the unique infinite word given by \Cref{Prop:Prop1}.  For $j \in \mathbb{N}$, define $\eta^{(j)}$ to be the infinite word associated to the sequence 
\begin{align*}
( \hspace{-0.5em}\underbrace{0, 0, \dots, 0}_{(N(j)-1) - \text{times}}\hspace{-0.5em}, l_{N(j)} - q(j), l_{N(j) +1}, l_{N(j) + 2}, \dots ),
\end{align*}
given by \Cref{Prop:Prop1}.
\end{definition}

\begin{remark}\label{rmk:remark_to_be_added}
Let $(l_{i})_{i \in \mathbb{N}}$ be a sequence of natural numbers.  The (generalised) Grigorchuk subshifts associated to the sequences $(0, 0, \dots, 0, l_{1}, l_{2}, l_{3}, \dots)$ and $(l_{1}, l_{2}, l_{3}, \dots)$ are topologically conjugate through the semi-group homomorphism which maps $a$ to $a$ and applies a cyclic permutation to $\{ x, y, z \}$.
\end{remark}

\begin{proof}[{Proof of \Cref{thm:Q-Bounds}}]
We structure the proof as follows.  We prove the following five statements from which we will deduce the required result.
\begin{enumerate}[itemsep=0.1em,topsep=-0.25em,label=(\roman*)]
\item\label{item1} $Q(2) = 2^{l_{1} +1} -1$
\item\label{item3a} If $k \in \mathbb{N}$ is such that $k \equiv 1 \pmod{4}$ or $k \equiv 3 \pmod{4}$, then $Q(k) = 1$.
\item\label{item3b} If $k \in \mathbb{N}$ is such that $k \equiv 2 \pmod{4}$ and $\eta\lvert_{k} \eta\lvert_{k} \in \mathcal{L}(\eta)$, then $\displaystyle Q(\eta\lvert_{k}) = \lfloor (2^{l_{1} + 2} - 2)/ k  \rfloor$.
\item\label{item3c}
If $k \in \mathbb{N}$ is such that $k \equiv 0 \pmod{4}$ and $\eta\lvert_{k} \eta\lvert_{k} \in \mathcal{L}(\eta)$, then
\begin{align}\label{eq:k=0mod4}
Q(\eta\lvert_{k}) = \left\lfloor \frac{2^{l_{N(j)}-q(j) + 1}-1}{k/2^{j+1}} \right\rfloor,
\end{align}
where $j$ is the smallest integer such that $k/2^{j} \equiv 2 \pmod{4}$.
\item\label{item2} Let $n \in \mathbb{N}$ and let $0 \leq r < 2^{n}$.  For each $v = (v_{1}, v_{2}, \dots, v_{2^{n} + r} )\in \mathcal{L}(\eta)$ with $Q(v) \geq 3$, there exists $1 \leq k \leq 2^{n} + r$ such that $\eta\lvert_{2^n+r} = (v_{k}, \ldots, v_{2^n+r}, v_{1}, \ldots ,v_{k-1})$ and, moreover, $Q(v) - 1 \leq Q(\eta\lvert_{2^n+r}) \leq Q(v)$.
\end{enumerate}
To prove Statement \ref{item1}, notice that $(y,a,y)$ and $(z,a,z)$ are not factors of $\eta$.  This follows, since each $(4k + 2)$-th letter of $\eta$ is equal to $x$, for all $k \in \mathbb{N}_{0}$.  By definition, we have that $\eta = \tau_{x}^{l_{1}} ( \eta^{(l_{1})})$.  Since the $(4k + 2)$-th letter of $\eta^{(l_{1})}$ is equal to $y$, for all $k \in \mathbb{N}_{0}$, it follows that $(x, a, x)$ is not a factor of $\eta^{(l_{1})}$, and hence, by \Cref{prop:length_of_tau},
\begin{align*}
Q((a, x))
= \frac{\lvert \tau^{l_{1}}_{x}(a) x \tau^{l_{1}}_{x}(a) \rvert - 1}{2}
= 2^{l_{1}+1} - 1.
\end{align*}
Since every second letter of $\eta$ is equal to $a$, it follows that if $n \equiv 1 \pmod{2}$, then $Q(n) = 1$.

Assume that the conditions of Statement \ref{item3b} hold, that is $k = 2^{n} + r \equiv 2 \pmod{4}$, where $n \in \mathbb{N}$ and $0 \leq r < 2^{n}$.  By construction we have that $\eta_{i} = x$ for all $i \equiv 2 \pmod{4}$.  Thus,
\begin{align*}
\raisebox{-1em}{
\begin{tikzpicture}
\draw(0,0)--(6.75,0);
\draw [dotted] (6.75,0)--(8,0);
\draw(8,0)--(11.4166,0);
\foreach \x in {0,1.75,2.25,4,4.5,6.25,6.75,8,9.75,10.25,11.4166}
\draw(\x,0.1)--(\x,-0.1);
\draw[decorate, decoration={brace}, yshift=1.5ex]  (0,0) -- node[above=0.4ex] {$(a, x, a)$}  (1.75,0);
\draw[decorate, decoration={brace}, yshift=1.5ex]  (2.25,0) -- node[above=0.4ex] {$(a, x, a)$}  (4,0);
\draw[decorate, decoration={brace}, yshift=1.5ex]  (4.5,0) -- node[above=0.4ex] {$(a, x, a)$}  (6.25,0);
\draw[decorate, decoration={brace}, yshift=1.5ex]  (8,0) -- node[above=0.4ex] {$(a, x, a)$}  (9.75,0);
\draw[decorate, decoration={brace}, yshift=1.5ex]  (10.25,0) -- node[above=0.4ex] {$(a, x)$}  (11.4166,0);
\draw[decorate, decoration={brace}, yshift=1.5ex]  (1.75,0) -- node[above=0.4ex] {$?$}  (2.25,0);
\draw[decorate, decoration={brace}, yshift=1.5ex]  (4,0) -- node[above=0.4ex] {$?$}  (4.5,0);
\draw[decorate, decoration={brace}, yshift=1.5ex]  (6.25,0) -- node[above=0.4ex] {$?$}  (6.75,0);
\draw[decorate, decoration={brace}, yshift=1.5ex]  (9.75,0) -- node[above=0.4ex] {$?$}  (10.25,0);
\node at (-0.4,-0.05) {$\eta\lvert_{k} = $};
\end{tikzpicture}}
\end{align*}
Since $\eta\lvert_{k} \eta\lvert_{k} \in \mathcal{L}(\eta)$, we have that $\eta\lvert_{k} = (a, x)^{k/2}$.  This in tandem with Statement \ref{item1} yields that
\begin{align*}
Q(\eta\lvert_{k}) = \left\lfloor \frac{2^{l_1+1}-1}{k/2} \right\rfloor.
\end{align*}

For Statement \ref{item3c}, notice that for all $j \in \mathbb{N}$ with $k\equiv 0 \pmod{2^{j}}$, we have
\begin{align*}
\eta\lvert_{k} = \tau^{(j)}(\eta^{(j)}\lvert_{k/2^{j}}).
\end{align*}
Since $\tau^{(j)}$ is a semi-group homomorphism on $\{a,x,y,z\}^{*}$, it follows that $Q(\eta\lvert_{k}) = Q(\eta^{(j)}\lvert_{k\vert2^{j}})$.  (Note here that $Q(\eta\lvert_{k})$ is taken with respect to the language $\mathcal{L}(\eta)$ and $Q(\eta^{(j)}\lvert_{k/2^{j}})$ is taken with respect to the language $\mathcal{L}(\eta^{(j)})$).  This in tandem with \Cref{rmk:remark_to_be_added} and Statement \ref{item3b} yields that 
\begin{align*}
Q(\eta\lvert_{k}) = Q(\eta^{(j)}\lvert_{k/2^{j}}) = \left\lfloor \frac{2^{l_{N(j)}-q(j) + 1}-1}{k/2^{j+1}} \right\rfloor,
\end{align*}
where $j$ is the smallest integer such that $k/2^{j} \equiv 2 \pmod{4}$.

We now turn to the proof of Statement \ref{item2}. By Statement \ref{item3a} it is sufficient to consider words of even length.  To this end, let $v \in \mathcal{L}(\eta)$ with $Q(v) \geq 3$ and with $\lvert v \rvert = 2^{n}+r$, for some $n\in\mathbb{N}$, and $0 \leq r < 2^{n}$.  Due to the structure of $\eta$ given in \eqref{eq:tau_structure_eta}, where we set $j = n -1$, and since every $(2m + 1)$-th question mark in \eqref{eq:tau_structure_eta} is equal to $\beta^{(j)}$, for all $m \in \mathbb{N}_{0}$, we have that $\tau^{(j)}((a, \beta^{(j)}))$ is a factor of $v^{Q(v)}$.  Thus there exists a natural number $k \leq 2^{n} + r$ such that
\begin{align*}
\eta\lvert_{2^{n}+r} = (v_{k}, \dots, v_{2^{n}+r}, v_{1}, \dots, v_{k-1}),
\end{align*}
see \eqref{eq:fig.1}, which yields that $Q(v)-1 \leq Q(\eta\lvert_{2^{n}+r}) \leq Q(v)$.
\begin{align}\label{eq:fig.1}
\raisebox{-2.45em}{
\begin{tikzpicture}
\draw(0.75,0)--(10.25,0);
\draw [dotted] (0.75,0)--(0,0);
\draw [dotted] (10.25,0)--(11,0);
\foreach \x in {1,4,7,10}
\draw(\x,0.2)--(\x,0);
\draw[decorate, decoration={brace}, yshift=2ex]  (1,0) -- node[above=0.4ex] {$v$}  (4,0);
\draw[decorate, decoration={brace}, yshift=2ex]  (4,0) -- node[above=0.4ex] {$v$}  (7,0);
\draw[decorate, decoration={brace}, yshift=2ex]  (7,0) -- node[above=0.4ex] {$v$}  (10,0);
\draw(2,0)--(2,-0.2);
\draw(5,0)--(5,-0.2);
\draw[decorate, decoration={brace, mirror}, yshift=-2ex]  (2,0) -- node[below=0.4ex] {$\eta\lvert_{2^{n}+r}$}  (5,0);
\draw[decorate, decoration={brace, mirror}, yshift=-2ex]  (5,0) -- node[below=0.4ex] {$\eta\lvert_{2^{n}+r}$}  (8,0);
\end{tikzpicture}}
\end{align}
With Statements \ref{item1}, \ref{item3a}, \ref{item3b},\ref{item3c} and \ref{item2} at hand we can now prove the required result.  If $k \equiv 0 \pmod{4}$, then the left hand side of \eqref{eq:k=0mod4} is maximised on the set $[2^{n}, 2^{n+1}) \cap \mathbb{N}$, at $j = n -1$, namely when $k = 2^{n}$.  Further, \eqref{eq:k=0mod4} in tandem with \eqref{eq:tau_structure_eta} and \Cref{prop:length_of_tau}, yields
\begin{align*}
Q(\eta\lvert_{2^{n}}) = 2^{l_{N(n-1)}-q(n-1) + 1} - 1.
\end{align*}
The function $n \mapsto Q(\eta\lvert_{2^{n}})$ is maximised on the set $[\sum_{i = 1}^{m} l_{i}, \sum_{i=1}^{m+1} l_{i}) \cap \mathbb{N}$ when $n - 1 = \sum_{i = 1}^{m} l_{i}$.  Indeed, we have that
\begin{align}\label{eq:explaination_lower}
Q(\eta\lvert_{k}) = 2^{l_{m+1}+1} - 1,
\end{align}
where $k = 2^{1 + \sum_{i = 1}^{m} l_{i}}$.  
Hence,
\begin{align*}
\limsup_{m\to\infty} \frac{2^{l_{m+1}+1} - 1}{{\left(2^{1 + \sum_{i=1}^{m} l_{i}}\right)}^{\alpha-1}}
\leq \limsup_{n\to\infty} \frac{Q(n)}{n^{\alpha-1}}
\leq \limsup_{n\to\infty} \frac{Q(\eta\lvert_{n}) + 1}{n^{\alpha-1}}
\leq \limsup_{m\to\infty} \frac{2^{l_{m+1}+1}}{{\left(2^{1 + \sum_{i=1}^{m} l_{i}}\right)}^{\alpha-1}}.
\end{align*}
Here 
the first inequality follows from \eqref{eq:explaination_lower};
the second inequality follows from the latter results of Statement \ref{item2}; 
the last inequality follows from Statements \ref{item1}, \ref{item3a}, \ref{item3b}, \ref{item3c} together with \eqref{eq:explaination_lower}.
\end{proof}

\begin{corollary}
An $l$-Grigorchuk subshift satisfies $Q(2^{j + 1}) = 2^{l_{N(j)} - q(j) + 1} - 1$, for all $j \in \mathbb{N}$.  
\end{corollary}

\begin{proof}
The result follows from \eqref{eq:tau_structure_eta}, \Cref{prop:length_of_tau} and Statement \ref{item3c} in the proof of \cref{thm:Q-Bounds} together with an argument by contradiction.
\end{proof}

\subsection{$\boldsymbol{\alpha}$-Repetitive}\label{sec:section4.3}

Our next result gives a necessary and sufficient condition on a given sequence of natural numbers $l = (l_{i})_{i = 1}^{\infty}$ to guarantee that the associated $l$-Grigorchuk subshift is $\alpha$-repetitive.  In particular, we obtain that an $l$-Grigorchuk subshift is $1$-repetitive if and only if $l$ is a bounded sequence. Thus, as $1$-repetitive implies linearly repetitive, if $l$ is a bounded sequence, then the associated $l$-Grigorchuk subshift is linearly repetitive.  We would also like to mention here that an exact formula for the repetitive function of an $l$-Grigorchuk subshift has been obtained, independently, in \cite{LenzSell:2016}, and hence they have also obtained a criterion similar to ours for an $l$-Grigorchuk subshift to be $\alpha$-repetitive.

\begin{theorem}\label{thm:G-alpha-repetative}
For $\alpha \geq 1$ an $l$-Grigorchuk subshift is $\alpha$-repetitive if and only if
\begin{align*}
\limsup_{n \to \infty} \left\lvert l_{n+2} + l_{n+1} + (1 - \alpha) \sum_{i = 1}^{n} l_{i} \right\rvert <\infty.
\end{align*}
\end{theorem}

We prove \Cref{thm:G-alpha-repetative} by using the following bounds on the repetitive function.

\begin{lemma}\label{lem:repetative_function_bounds}
Let $l = (l_{i})_{i \in \mathbb{N}}$ denote a sequence of natural numbers. The repetitive function for an $l$-Grigorchuk subshift satisfies the following inequalities, for $j \in \mathbb{N}$,
\begin{align*}
2^{l_{N(j) + 1} + l_{N(j)} - q(j) + j + 1} \leq R\left( 2^{j + 1} - 1 \right)
\leq 2^{l_{N(j) + 1} + l_{N(j)} - q(j) + j + 2}.
\end{align*}
\end{lemma}

\begin{proof}
By \eqref{eq:tau_structure_eta} we have that $\tau^{(j-1)} \circ \tau_{x} (a)$, $\tau^{(j-1)} \circ \tau_{y} (a)$ and $\tau^{(j-1)} \circ \tau_{z} (a)$ all belong to $\mathcal{L}(\eta)$ and that
\begin{align*}
\left\lvert \tau^{(j-1)} \circ \tau_{x} (a) \right\rvert =
\left\lvert \tau^{(j-1)} \circ \tau_{y} (a) \right\rvert =
\left\lvert \tau^{(j-1)} \circ \tau_{z} (a) \right\rvert =
\left \lvert \tau^{(j)}(a) \right\rvert =
2^{1 + q(j) + \sum_{i = 1}^{N(j) - 1} l_{i}} - 1 = 2^{j + 1} - 1,
\end{align*}
We claim that, for all $k \in \{ 1, 2, \dots, l_{N(j)} - q(j)\}$, the word
\begin{align*}
\tau^{(j + k)}(a) = \tau^{(j)} \circ \tau_{\beta^{(j)}}^{k}(a)
\end{align*}
does not contain as factors both the words 
\begin{align}\label{eq:factors}
\tau^{(j-1)} \circ \tau_{\beta^{(j + l_{N(j)} - q(j))}}(a)
\quad \text{and} \quad
\tau^{(j-1)} \circ \tau_{\beta^{(j + l_{N(j)} - q(j) + l_{N(j) + 1})}}(a).
\end{align}
For if this were the case, then, since the first letter of the words in \eqref{eq:factors} is equal to $a$ and both $\tau^{(j-1)}(a)$ and $\tau^{(j)}(a)$ are palindromes, there exists an integer $m \in [2^{j-1}+1, 2^{j}-1]$ with
\begin{align}\label{eq:repetative_x_z1}
\sigma^{2m}(\tau^{(j)}(a) \beta^{(j)} \tau^{(j)}(a)) \lvert_{2^{j+1}-1} = \tau^{(j-1)}(a) \beta^{(j + l_{N(j)} - q(j))} \tau^{(j-1)}(a)
\end{align}
or, such that
\begin{align}\label{eq:repetative_x_z2}
\sigma^{2m}(\tau^{(j)}(a) \beta^{(j)} \tau^{(j)}(a)) \lvert_{2^{j+1}-1} = \tau^{(j-1)}(a) \beta^{(j + l_{N(j)} - q(j) + l_{N(j) + 1})} \tau^{(j-1)}(a).
\end{align}
Thus, the $(2^{j+1}-2m)$-th letter of $\tau^{(j-1)}(a)$ is equal to $\beta^{(j)}$ and the $(2m - 2^{j})$-th letter of $\tau^{(j-1)}(a)$ is equal to $\beta^{(j + l_{N(j)} - q(j))}$ in the case of \eqref{eq:repetative_x_z1} and $\beta^{(j + l_{N(j)} - q(j) + l_{N(j) + 1})}$ in the case of \eqref{eq:repetative_x_z2}.  As $\tau^{(j-1)}(a)$ is a palindrome, $\beta^{(j)} \neq \beta^{(j + l_{N(j)} - q(j))}$ and $\beta^{(j)} \neq \beta^{(j + l_{N(j)} - q(j)+ l_{N(j) + 1})}$, this yields a contradiction to the initial assumption.

Similarly, for all $k \in \{1, 2, \dots, l_{N(j)+1} \}$, the word
\begin{align*}
\tau^{(j + l_{N(j)} - q(j) + k)}(a) = \tau^{(j)} \circ \tau^{l_{N(j)} - q(j)}_{\beta^{(j)}} \circ \tau^{k}_{\beta^{(j + l_{N(j)} - q(j))}}(a)
\end{align*}
does not contain as a factor the word 
\begin{align*}
\tau^{(j-1)} \circ \tau_{\beta^{(j + l_{N(j)} - q(j)+ l_{N(j) + 1})}}(a).
\end{align*}
This yields the lower bound for the repetitive function, namely that
\begin{align*}
R(2^{j + 1} - 1) \geq \left\lvert \tau^{(j + l_{N(j)} - q(j) + l_{N(j) + 1})}(a) \right\rvert + 1 = 2^{j + l_{N(j)} - q(j) + l_{N(j) + 1}}.
\end{align*}
Due to the structure of $\eta$, given a word of length $2^{j+1}-1$ in $\mathcal{L}(\eta)$ it is necessarily a factor of $\tau^{(j)} \circ \tau_{x}(a)$, $\tau^{(j)} \circ \tau_{y}(a)$ or $\tau^{(j)} \circ \tau_{z}(a)$.  Thus, any word of length $2^{j+1}-1$ is a factor of
\begin{align*}
\tau^{(j + l_{N(j)} - q(j) + l_{N(j) + 1} + 1)}(a) = \tau^{(j)} \circ \tau^{l_{N(j)} - q(j)}_{\beta^{(j)}} \circ \tau^{l_{N(j)}}_{\beta^{(j + l_{N(j)} - q(j))}} \circ \tau_{\beta^{(j + l_{N(j)} - q(j)+ l_{N(j) + 1})}} (a).
\end{align*}
This in tandem with \eqref{eq:tau_structure_eta} and \Cref{prop:length_of_tau} yields that
\begin{align*}
R(2^{j + 1} - 1)
\leq 2 \left\lvert \tau^{(j + l_{N(j)} - q(j) + l_{N(j) + 1} + 1)}(a) \right\rvert
< 2^{j + l_{N(j)} - q(j) + l_{N(j) + 1} + 2},
\end{align*}
which completes the proof.
\end{proof}

\begin{proof} [{Proof of \Cref{thm:G-alpha-repetative}}]
For $n \in \mathbb{N}$, let $j = j(n)$ denote the unique natural number such that $2^{j - 1} \leq n < 2^{j}$.  By definition, the repetitive function is monotonically increasing, and so $R(2^{j-1} - 1) \leq R(n) \leq R(2^{j} - 1)$.  Combining this with \Cref{lem:repetative_function_bounds}, yields that
\begin{align*}
2^{1 - \alpha} 2^{l_{N(j-1)} - q(j) + l_{N(j-1)+1} - (j - 1) (\alpha - 1)}
\leq \frac{R(n)}{n^{\alpha}}
\leq 2^{2 + \alpha} 2^{l_{N(j)} - q(j) + l_{N(j) + 1} - j(\alpha - 1)}.
\end{align*}
Since
\begin{align*}
l_{N(j)} - q(j) + l_{N(j) + 1} - j(\alpha - 1) \leq l_{N(j)} + l_{N(j) + 1} - (\alpha - 1) \sum_{k = 1}^{N(j) - 1} l_{k},
\end{align*}
we have that $0 < R_{\alpha} < \infty$, if and only if,
\begin{align*}
\limsup_{j \to \infty} \left\lvert l_{N(j)} + l_{N(j) + 1} - (\alpha - 1) \sum_{k = 1}^{N(j) - 1} l_{k} \right\rvert < \infty.
\end{align*}
This completes the proof.
\end{proof}

\subsection{Examples}
\label{sec:examples}

Here we discuss several examples of sequences $l = (l_{n})_{n \in \mathbb{N}}$ for which the associated $l$-Grigorchuk subshift exhibts difference order characteristics.

\begin{example}\label{ex:example} \mbox{ }
\begin{enumerate}[itemsep=0.1em,topsep=-0.25em,label=(\roman*)]
\item\label{(1)} If $l$ is a bounded sequence, then the associated $l$-Grigorchuk subshift is $1$-repetitive and $1$-repulsive, and hence, $1$-finite.
\item\label{(3)} Let $b \geq 2$ denote a fixed integer.  If $l = ( b^{n} )_{n \in \mathbb{N}}$, then the associated $l$-Grigorchuk subshift is $b^{2}$-repetitive, and $b$-repulsive (and hence $b$-finite).  Thus, in general, $\alpha$-repetitive is not equivalent to $\alpha$-repulsive, and hence nor $\alpha$-finite.
\item\label{(4)} Let $( b_{n} )_{n \in \mathbb{N}}$ denote a bounded sequence, and set $l_{n} = 2^{n/2} - b_{n/2}$ if $n$ is even, and set $l_{n} = b_{(n+1)/2}$ otherwise.  The associated $l$-Grigorchuk subshift is $2$-repetitive, however, it is not $\alpha$-repulsive nor $\alpha$-finite, for any value of $\alpha \geq 1$.
\item Let $l_{n} = 2^{n/2}-n$ if $n$ is even and $l_{n} = n$ otherwise. The associated $l$-Grigorchuk subshift is neither $\alpha$-repetitive, $\alpha$-repulsive nor $\alpha$-finite for any value of $\alpha \geq 1$.
\item If $l = (l_{n})_{n \in \mathbb{N}}$ is a sequence of natural number such that there exists a non-constant polynomial $P$ with $ l_{n}=P(n)$, then the $l$-Grigorchuk subshift is neither \mbox{$\alpha$-repulsive}, \mbox{$\alpha$-finite} nor $\alpha$-repetitive, for any value of $\alpha \geq 1$.  This is a consequence of Faulhalber's formula \cite{CG:95}.
\end{enumerate}
\end{example}

From \Cref{ex:example}~\ref{(3)} and~\ref{(4)}, for $\alpha> 1$, we see that the $l$-Grigorchuk subshifts provide examples which demonstrate that $\alpha$-repulsive, and hence $\alpha$-finite, is not equivalent to \mbox{$\alpha$-repetitive}.  This gives rise to the question how the notions of $\alpha$-repetitive and $\beta$-repulsive, and hence $\beta$-finite, are connected in terms of $l$-Grigorchuk subshifts. This is what we address in the following proposition; indeed the connection, which we have observed in \Cref{ex:example}~\ref{(1)} and~\ref{(3)} is in fact true in general.

\begin{proposition}
Let $l$ be a sequence of natural numbers. If the $l$-Grigorchuk subshift is $\alpha$-repulsive, and hence $\alpha$-finite, then it is $\alpha^2$-repetitive.
\end{proposition}

\begin{proof}
Observe that, for all $n \in \mathbb{N}$,
\begin{align}\label{eq:bounds_equivalence_1}
l_{n+2} + (1 - \alpha) \sum_{i = 1}^{n+1} l_{i}
= l_{n+2} + l_{n+1}+ \left(1 - \alpha \left( 1+\frac{l_{n+1}}{\sum_{i = 1}^{n} l_{i}} \right) \right)\sum_{i = 1}^{n} l_{i}.
\end{align}
By the hypothesis and \Cref{thm:G-alpha-free}, we have that $\limsup_{n \to \infty} \lvert l_{n+1} + (1 - \alpha) \sum_{i = 1}^{n} l_{i} \rvert$ is a finite real number. In the following, we denote this value by $c$.  Given $\epsilon >0$, there exists an $N \in \mathbb{N}$, such that, for all $n\geq N$,
\begin{align*}
-c - \epsilon \leq l_{n+1} + (1 - \alpha) \sum_{i = 1}^{n} l_{i} \leq c + \epsilon,
\quad \text{and hence,} \quad
\alpha - \frac{ c + \epsilon}{\sum_{i=1}^{n}l_i} \leq 1 + \frac{l_{n+1}}{\sum_{i=1}^{n}l_i} \leq \alpha + \frac{ c + \epsilon}{\sum_{i=1}^{n}l_i}.
\end{align*}
This in tandem with \Cref{eq:bounds_equivalence_1} yields for $\delta\geq 1$ that
\begin{align*}
-\delta(c+\epsilon) + l_{n+2} + (1 - \delta) \sum_{i = 1}^{n+1} l_{i} 
\leq l_{n+2} +l_{n+1} + (1 - \delta \alpha) \sum_{i = 1}^{n} l_{i} 
\leq \delta(c+\epsilon) + l_{n+2} + (1 - \delta) \sum_{i = 1}^{n+1} l_{i},
\end{align*}
for all $n\geq N$. This in combination with the hypothesis of the proposition and the \Cref{thm:G-alpha-free,thm:G-alpha-repetative} yields the required result.
\end{proof}

\subsection{Aperiodicity, Complexity and Ergodicity}\label{sec:section4.4}

We now turn to computing the value of the complexity function for a given $l$-Grigorchuk subshift.  Knowing the behaviour of the complexity function allows us to conclude that any $l$-Grigorchuk subshift is aperiodic and uniquely ergodic.  We note that in \cite{LenzSell:2016} an explicit formula for the complexity and the palindromic complexity functions have also been obtained independently.  The proof of the following theorem is a generalisation of that given in \cite{GLN:16,GLN:16b}, where the case when $l$ is the constant one sequence is considered.

In the sequel, for ease of notation, for $n \in \mathbb{N}_{0}$, we set $M(n)\coloneqq \lvert\tau^{(\sum_{i=1}^{n} l_i )} (a) \rvert=2^{1+\sum_{i=1}^{n} l_i }-1$.

\begin{theorem}\label{thm:complexity}
For $m \in \mathbb{N}_{0}$ and $0 \leq r < M({m+1}) - M(m)$, the $l$-Grigorchuk subshift satisfies,
\begin{align*}
p(M(m) + 1 + r) = 
\begin{cases}
2M(m)+M({m-1}) +3r & \mbox{if} \; 0 \leq r< M(m)- M({m-1}),\\
3M(m) +2r & \mbox{if} \; M(m) - M(m-1) \leq r < M(m+1) - M(m).
\end{cases} 
\end{align*} 
\end{theorem}

For the proof of this result we will use the following lemma.

\begin{lemma}
\label{lem:lemma534}
The factor $\tau^{(j)}(a)$ is $3$-special for every $j\in\mathbb{N}_0$.
\end{lemma}

\begin{proof}
This follows from the structure of $\eta$ given in \eqref{eq:tau_structure_eta}.
\end{proof}

\begin{proof}[{Proof of \Cref{thm:complexity}}]
For $m=1$, every word of length $\lvert \tau^{( l_{1} )}(a) \rvert + 1$ in $\mathcal{L}(\eta)$ is a factor of at least one of the following words belonging to $\mathcal{L}(\eta)$.
\begin{align*}
\tau^{(l_1) }(a) x \tau^{(l_1)}(a)& = (\underbrace{a, x, a, \dots, a, x, a}_{\tau^{(l_1) }(a)}, x, \underbrace{a, x, a, \dots, a, x, a}_{\tau^{(l_1)}(a)})\\
\tau^{(l_1)  }(a) y \tau^{(l_1)}(a)& = (\underbrace{a, x, a, \dots, a, x, a}_{\tau^{(l_1) }(a)}, y, \underbrace{a, x, a, \dots, a, x, a}_{\tau^{(l_1)}(a)})\\
\tau^{(l_1)  }(a) z \tau^{(l_1)}(a)& = (\underbrace{a, x, a, \dots, a, x, a}_{\tau^{(l_1) }(a)}, z, \underbrace{a, x, a, \dots, a, x, a}_{\tau^{(l_1)}(a)})
\end{align*}
This yields that $p( \lvert \tau^{(l_1) }(a) \rvert + 1)=2 \lvert \tau^{(l_1)}(a) \rvert + \lvert \tau^{(0)}(a) \rvert$.  In the same way, for a fixed $m \in \mathbb{N}$, every word of length $\lvert \tau^{(\sum_{i=1}^{m} l_i )}(a) \rvert + 1$ in $\mathcal{L}(\eta)$ is a factor of at least one of the following words
\begin{align*}
\tau^{(\sum_{i=1}^{m} l_i )}(a)  x  \tau^{(\sum_{i=1}^{m} l_i )}(a), \quad
\tau^{(\sum_{i=1}^{m} l_i )}(a)  y  \tau^{(\sum_{i=1}^{m} l_i )}(a) \quad \text{and} \quad
\tau^{(\sum_{i=1}^{m} l_i )}(a)  z  \tau^{(\sum_{i=1}^{m} l_i )}(a),
\end{align*}
which are all contained in $\mathcal{L}(\eta)$ by \eqref{eq:tau_structure_eta}. Additionally, we have 
\begin{align*}
\tau^{(\sum_{i=1}^{m} l_i )}(a)  \beta^{(\sum_{i=1}^{m- 1} l_i)}  \tau^{(\sum_{i=1}^{m} l_i )}(a)
= \tau^{(\sum_{i=1}^{m} l_i )}(a)\underbrace{
\underbrace{\beta^{(\sum_{i=1}^{m- 1} l_i)} \tau^{(\sum_{i=1}^{m-1} l_i )}(a)}
\cdots
\underbrace{\beta^{(\sum_{i=1}^{m- 1} l_i)}
\tau^{(\sum_{i=1}^{m-1} l_i )}(a)}}_{2^{l_{m}} - \text{times}}.
\end{align*}
With this we obtain that, for all $m\in\mathbb{N}$,
\begin{align}
\label{eq:mylab357}
p( \lvert \tau^{(\sum_{i=1}^{m} l_i )}(a) \rvert + 1) \leq 2 \lvert \tau^{(\sum_{i=1}^{m} l_i )}(a) \rvert +\lvert \tau^{(\sum_{i=1}^{m-1} l_i )}(a) \rvert.
\end{align}
By Lemma \ref{lem:lemma534} the factor $\tau^{(\sum_{i=1}^{m} l_i )}(a)$ is $3$-right special, for all $m\in\mathbb{N}$, and so
\begin{align*}
\tau^{(1+\sum_{i=1}^{m} l_i)}(a)=\tau^{(\sum_{i=1}^{m} l_i )}(a)\beta^{(\sum_{i=1}^{m} l_{i})}\tau^{(\sum_{i=1}^{m} l_i )}(a)
\end{align*}
is $3$-right special as it is a suffix of $\tau^{(\sum_{i=1}^{m+1} l_i )}(a)$.  Notice that 
\begin{align*}
\tau^{(\sum_{i=1}^{m} l_i )}(a)\beta^{(\sum_{i=1}^{m-1} l_{i})}\tau^{(\sum_{i=1}^{m} l_i )}(a),
\end{align*}
has the same length as $\tau^{(1+\sum_{i=1}^{m} l_i )}(a)$, but it is not right special because, by \eqref{eq:tau_structure_eta}, the only possible right-extension is 
\begin{align*}
\tau^{(\sum_{i=1}^{m} l_i )}(a)\beta^{(\sum_{i=1}^{m-1} l_{i})}\tau^{(\sum_{i=1}^{m} l_i )}(a)\beta^{(\sum_{i=1}^{m} l_{i})}.
\end{align*}
However, due to the structure of $\eta$ given in \Cref{Prop:Prop1} and \eqref{eq:tau_structure_eta}, the prefix
\begin{align*}
\tau^{(\sum_{i=1}^{m} l_i )}(a)\underbrace{
\underbrace{\beta^{(\sum_{i=1}^{m- 1} l_i)}\tau^{(\sum_{i=1}^{m-1} l_i )}(a)}
\cdots
\underbrace{\beta^{(\sum_{i=1}^{m- 1} l_i)}\tau^{(\sum_{i=1}^{m-1} l_i )}(a)}}_{(2^{l_{m}} - 1) - \text{times}},
\end{align*}
whose length is equal to $2 \lvert \tau^{(\sum_{i=1}^{m} l_i )}(a)\rvert - \lvert\tau^{(\sum_{i=1}^{m-1} l_i )}(a)\rvert$, is $2$-right special.  Further, it is not a suffix of $\tau^{(\sum_{i=1}^{m+1} l_i )}(a)$.  Using these right special words and their respective suffixes of length strictly greater than $\lvert \tau^{(\sum_{i=1}^{m} l_i )}(a) \rvert$ we obtain that 
\begin{align}\label{eq:mylab358}
\hspace{-0.35em}p(\lvert \tau^{(\sum_{i=1}^{m+1} l_i )}(a) \rvert + 1) - p( \lvert \tau^{(\sum_{i=1}^{m} l_i )}(a) \rvert + 1) \geq 2 \lvert \tau^{(\sum_{i=1}^{m+1} l_i )}(a) \rvert - \lvert \tau^{(\sum_{i=1}^{m} l_i )}(a) \rvert - \lvert \tau^{(\sum_{i=1}^{m-1} l_i )}(a) \rvert.
\end{align}
The result follows by combining $\eqref{eq:mylab357}$ and \eqref{eq:mylab358} together with an inductive argument.
\end{proof}

\begin{corollary}\label{cor:Gri-aperiodic}
Every $l$-Grigorchuk subshift is aperiodic.
\end{corollary}

\begin{proof}
 By \Cref{prop:minimality} we know that every $l$-Grigorchuk subshift is minimal.  Therefore, if an $l$-Grigorchuk subshift was not aperiodic, then its complexity function would be bounded, contradicting \Cref{thm:complexity}.
\end{proof}

\begin{corollary}\label{Cor:unique_ergodic}
Every $l$-Grigorchuk subshift is uniquely ergodic.
\end{corollary}

\begin{proof}
Given an $l$-Grigorchuk subshift $\Omega(\eta)$ we define the associated two-sided subshift $\Omega'(\eta)$ by $\Omega'(\eta) \coloneqq \{ \omega \in \{ a, x, y, z \}^{\mathbb{Z}} \colon \mathcal{L}(\omega) \subseteq \mathcal{L}(\eta) \}$.  Here $\{ a, x, y, z \}^{\mathbb{Z}}$ denotes the set of all bi-infinite words over the alphabet $\{ a, x, y, z \}$ equipped with the discrete product topology.  Since $\eta$ is uniformly recurrent (see \Cref{prop:minimality}), we have that $\Omega'(\eta)$ is minimal.  (For the latter result, see for instance \cite{Combinatorics:2010}.)  The existence of an invariant measure supported on $\Omega'(\eta)$ is guaranteed by \cite{B:1984/85}.  By \Cref{lem:lemma534,thm:complexity} and \cite[Theorem 2.2]{B:1984/85}, where in this latter result we set $\alpha = 4$ and $k = 1$, it follows that $\Omega'(\eta)$ has at most one ergodic measure $\mu$. Therefore, $(\Omega'(\eta),\sigma)$ is a uniquely ergodic dynamical system. Since as a dynamical system,  $(\Omega(\eta),\sigma)$ is a topological factor of $(\Omega'(\eta),\sigma)$ via the factor map $\pi \colon \Omega'(\eta)\to \Omega(\eta)$ given by $\pi(\dots, x_{-2}, x_{-1}, x_{0}, x_{1}, x_{2}, \dots) = (x_{1}, x_{2}, \dots)$, it follows that also $(\Omega(\eta),\sigma)$ is uniquely ergodic. To see this fix a continuous function $f\colon\Omega(\eta)\to \mathbb{R}$ and $x\in \Omega(\eta)$. Then there exists $y\in \Omega'(\eta)$ with $x=\pi(y)$ and we have
\[
\lim_{n\to\infty}\frac{1}{n}\sum_{k=0}^{n-1}f\circ \sigma^{k}(x)=
\lim_{n\to\infty}\frac{1}{n}\sum_{k=0}^{n-1}f\circ \sigma^{k}\circ \pi(y)=
\lim_{n\to\infty}\frac{1}{n}\sum_{k=0}^{n-1}f\circ  \pi\circ\sigma^{k}(y)=\int f \;\mathrm{d}\mu\circ\pi^{-1}.
\]
This characterises  unique ergodicity as stated e.g. in  \cite[Theorem 6.19]{W:2013}. 
\end{proof}

Alternatively, one can show that any $l$-Grigorchuk subshift is a regular Toeplitz subshift, and so it is uniquely ergodic, see \cite{PK:2003}.

\begin{remark}\label{rmk:last_remark}
In most sections of this article, we assumed that $l_i\ne 0$ for all $i\in\mathbb{N}$. We believe that all of our results hold under slightly weaker assumptions, namely that if $l_{i} = 0$, for some index $i$, then $l_{i-1}$ and $l_{i+1}$ are non-zero, and the homomorphisms $\tau_x$, $\tau_y$ and $\tau_z$ all occur infinitely often in the construction of $\eta$.
\end{remark}

\end{document}